\documentclass[10pt]{article}

\usepackage{amsmath,amscd,amsfonts,amsthm}
 \usepackage{graphicx}
 
\newcommand{\shrinkmargins}[1]{
  \addtolength{\textheight}{#1\topmargin}
  \addtolength{\textheight}{#1\topmargin}
  \addtolength{\textwidth}{#1\oddsidemargin}
  \addtolength{\textwidth}{#1\evensidemargin}
  \addtolength{\topmargin}{-#1\topmargin}
  \addtolength{\oddsidemargin}{-#1\oddsidemargin}
  \addtolength{\evensidemargin}{-#1\evensidemargin}
  }

\shrinkmargins{.7}

\DeclareMathOperator{\GL}{GL}
\DeclareMathOperator{\Sp}{Sp}
\DeclareMathOperator{\GSp}{GSp}

\DeclareMathOperator{\Frob}{Frob}

\DeclareMathOperator{\Aut}{Aut}
\DeclareMathOperator{\Out}{Out}
\DeclareMathOperator{\Epi}{Epi}

\DeclareMathOperator{\Jac}{Jac}

\newcommand{\field}[1]{\mathbb{#1}}

\newcommand{\Qp}{\field{Q}_p}
\newcommand{\G}{\mathbf{G}}
\newcommand{\E}{\mathbb{E}}
\newcommand{\cc}{\mathbf{c}}
\newcommand{\Q}{\field{Q}}

\newcommand{\Zp}{\field{Z}_p}

\newcommand{\Z}{\field{Z}}
\newcommand{\F}{\field{F}}

\newcommand{\Fqbar}{\bar{\field{F}}_q}

\newcommand{\CC}{\mathcal{C}}
\newcommand{\ra}{\rightarrow}

\newcommand{\beq}{\begin{displaymath}}
\newcommand{\eeq}{\end{displaymath}}
\newcommand{\beqn}{\begin{equation}}
\newcommand{\eeqn}{\end{equation}}

\theoremstyle{plain}
\newtheorem{thm}{Theorem}[section]
\newtheorem{prop}[thm]{Proposition}

\newtheorem{lem}[thm]{Lemma}
\newtheorem{heur}[thm]{Heuristic}

\theoremstyle{definition}
\newtheorem{defn}[thm]{Definition}

\theoremstyle{remark}
\newtheorem{rem}[thm]{Remark}

\title{Random pro-$p$ groups, braid groups, and random tame Galois groups}
\author{Nigel Boston \thanks{Partially
supported by the Claude
Shannon Institute, Science
Foundation Ireland Grant 06/MI/006 and Stokes Professorship award, Science Foundation Ireland Grant 07/SK/I1252b. We thank Joann Boston for making the figure.} and Jordan S. Ellenberg \thanks{Partially supported by NSF-CAREER Grant DMS-0448750 and a Sloan research fellowship.}}

\begin{document}

\maketitle
\section{Introduction}

In this article we introduce a heuristic prediction for the distribution of the isomorphism class of $G_S(p)$, the Galois group of the maximal pro-$p$ extension of $\Q$ unramified outside of $S$, where $S$ is a "random" set of primes.  That these groups should exhibit any statistical regularity is not at all obvious; our expectations in this direction are guided by the Cohen-Lenstra conjectures, which (among other things) predict quite precisely how often a fixed finite group appears as the ideal class group of a quadratic imaginary field.

The Cohen-Lenstra conjectures can be obtained in (at least) two ways.  On the one hand, the distribution on finite abelian groups suggested by the heuristics has a good claim on being the most natural ``uniform distribution" on the category of finite abelian $p$-groups.  On the other hand, as observed by Friedman and Washington (\cite{frie:friedmanwashington}, see also \cite{acht:cohenlenstra}) the conjectures can also be recovered via the analogy between number fields and function fields; here one thinks of the class group as the cokernel of $\gamma-1$ where $\gamma$ is a $p$-adic matrix drawn randomly from a suitable subset of the $\Q_p$-points of an algebraic group.  We will show that both heuristic arguments can be generalized to the nonabelian pro-$p$ case, and that both lead to the same prediction, Heuristic~\ref{he:main} below.

We conclude by describing some evidence, both theoretical and experimental, that supports (or at least is consistent with) Heuristic~\ref{he:main}.   We pay special attention to the interesting case where $p=2$ and $S$ consists of two primes congruent to $5 \pmod 8$.  In this case, the heuristic appears to suggest that $G_S(p)$ is infinite $1/16$ of the time.  

\subsection{Notation}
When $x$ and $y$ are elements of a group, we use $x \sim y$ to mean "$x$ is conjugate to $y$."  We say a pro-$p$ group $\Gamma$ is {\em balanced} if its generator rank equals its relator rank.

\section{Statement of the heuristic}

If $S$ is a set of primes in $\Q$, we denote by $G_S(p)$ the Galois group of the maximal $p$-extension unramified away from $S$ (including infinity if $p=2$).  Our aim in this section is to present a heuristic answer to the question:  "When $S$ is a random set of primes, what is the probability that $G_S(p)$ is isomorphic to some specified finite $p$-group $\Gamma$?"

In order to phrase the question precisely, we need a little notation.

If $S = (\ell_1, \ldots, \ell_g)$ is a $g$-tuple of primes congruent to $1$ mod $p$, we denote by $Z_i$ the closure of $\ell_i^\Z$ in $\Z_p^*$ and by $W_i$ the group $\Z_p / (\ell_i - 1)\Z_p$.  Note that $G_S(p)^{ab}$ is isomorphic to $W = \oplus_{i=1}^g W_i$, a finite abelian $p$-group of rank $g$.   In this paper, $S$ always denotes an {\em ordered} $g$-tuple of primes.

\begin{defn} The  {\em type} of $S$ is the sequence of subgroups $(Z_1, \ldots, Z_g)$. 
\end{defn}

Note that when $p$ is odd, the type of $(\ell_1, \ldots, \ell_g)$ is determined by the maximal power of $p$ dividing $\ell_i-1$ for each $i$; in particular, the type of $S$ carries the same information as the sequence of groups $W_i$.  When $p = 2$, the type of $S$ determines the $W_i$, but is not determined by it; for instance, primes $\ell$ which are  $3 \pmod 8$ and those which are $7 \pmod 8$ are of different types, but both have $\Z_2/(\ell-1)\Z_2 \cong \Z/2\Z$.   We write $W(Z)$ for the finite $p$-group attached to a type $Z = (Z_1, \ldots, Z_g)$ by taking the sum of the corresponding $W_i$. 

Let $Z = (Z_1, \ldots, Z_g)$ be a type, and $\Gamma$ a finite $p$-group such that $\Gamma^{ab} \cong W(Z)$ and $h_1(\Gamma,\F_p) = h_2(\Gamma,\F_p) = g$, where $h_i(\Gamma,\F_p)$ denotes the dimension of $H^i(\Gamma,\F_p)$.   Then $\Gamma$ is balanced. We define $P(Z,\Gamma,X)$ to be the proportion of $g$-tuples of primes $S = (\ell_1, \ldots, \ell_g)$ with type $Z$ in $[X, ..., 2X]^g$ such that $G_S(p) \cong \Gamma$.  Then the behavior of $P(Z,\Gamma,X)$ as $X$ grows can be thought of as the probability that a random $g$-tuple of primes of type $Z$ has $\Gamma$ as its maximal unramified Galois group.

In order to state our heuristic estimate for this probability, we need a little more notation.

Write $\pi: \Gamma \ra \Gamma^{ab}$ for the natural projection. 

\begin{defn}
Let $A_Z(\Gamma)$ be the number of pairs $((c_1, \ldots, c_g),\iota)$ where $(c_1, \ldots, c_g)$ is a $g$-tuple of conjugacy classes in $\Gamma$ and $\iota$ is an involution in $\Gamma$ (automatically trivial when $p$ is odd) such that:
\begin{itemize}
\item $c_i^z = c_i$ for all $z \in Z_i$;
\item the elements $\pi(c_1), \ldots, \pi(c_g)$ generate $\Gamma^{ab}$;
\item the map $W(Z) \ra \Gamma^{ab}$ sending $(w_1, \ldots, w_g)$ to $\sum_i w_i \pi(c_i)$ is an isomorphism;
\item if $p = 2$, and $\iota_i$ is the unique nontrivial involution in the cyclic subgroup of $\Gamma^{ab}$ generated by $\pi(c_i)$, we have $\sum_{i=1}^g \iota_i = \pi(\iota)$.
\end{itemize}
\label{de:azg}
\end{defn}

\begin{rem} When no confusion is likely (i.e. when we are restricting attention to a particular type $Z$) we will write $A(\Gamma)$ for $A_Z(\Gamma)$.
\end{rem}

\begin{heur}
$\lim_{X \ra \infty} P(Z,\Gamma,X)$ exists and is equal to $A_Z(\Gamma)/|\Aut(\Gamma)|$.
\label{he:main}
\end{heur}

The value of $A_Z(\Gamma)/|\Aut(\Gamma)$ is easy to compute for any particular choice of $Z$ and $\Gamma$.
For example, in \S \ref{s:evidence}, we discuss the case where $p$ is an odd prime, $\Gamma$ is a $2$-generator $2$-relator group of order $p^3$ with abelianization $(\Z/p\Z)^2$, and $Z_1 = Z_2 = 1+p\Z_p$.  In this case, we show that $A_Z(\Gamma)/\Aut(\Gamma) = 1 - p^{-2}$, and we in fact show that Heuristic~\ref{he:main} is correct in this case.

We note that the sum over all finite $\Gamma$ of $A_Z(\Gamma)/|\Aut(\Gamma)|$ need not be equal to $1$.  This should correspond to the fact that there may be a positive probability that $G_S(p)$ is infinite.  Certainly, the philosophy of the paper demands that the sum of $A_Z(\Gamma)/|\Aut(\Gamma)|$ should be {\em at most} $1$, but we do not at present know how to prove this inequality.  Best of all would be to devise a suitable extension of the heuristics discussed here to infinite pro-$p$ groups $\Gamma$.

In the following sections, we explain two justifications for Heuristic~\ref{he:main}, each one adapted from a justification for the Cohen-Lenstra heuristics.



\section{Justification 1:  random $p$-groups with inertia data}

There is no uniform distribution on a countably infinite set.  Nonetheless, there is a natural ``uniform" distribution on the set of isomorphism classes of finite abelian $p$-groups, following the usual principle that objects in a category should be assigned a weight inversely proportional to the order of their automorphism group.  (See Leinster\cite{le:category} for a thorough development of this idea.) The sum of $|\Aut(\Gamma)|^{-1}$ as $\Gamma$ ranges over isomorphism classes of finite abelian $p$-groups is finite; thus, there is a unique probability distribution on these isomorphism classes such that the mass assigned to $\Gamma$ is proportional to $|\Aut(\Gamma)|^{-1}$.  The Cohen-Lenstra heuristic can be thought of as asserting that the $p$-part of the class group of a random imaginary quadratic field is a ``random" abelian group, in the sense that it follows this ``uniform" distribution on finite abelian $p$-groups.  Many of the more general conjectures of Cohen, Lenstra, and Martinet are in the same spirit, predicting that class groups of extensions with more complicated structure (e.g. Galois extensions with Galois group $K$) are ``random" objects in more general categories (e.g. the category of finitely generated $\Z/p\Z[K]$-modules) in the analogous sense.

We now present an argument in a similar spirit that leads us to Heuristic~\ref{he:main}.

One might start by trying to construct a probability distribution on finite $p$-groups where the measure assigned to $\Gamma$ is proportional to $|\Aut(\Gamma)|^{-1}$.  This, however, is easily seen to be incorrect.  The relevant difference between the Cohen-Lenstra context and the one treated in the present paper is that, as far as we know, any finite abelian $p$-group can arise as the $p$-part of the class group of an imaginary quadratic field.  By contrast, not every balanced $p$-group can be $G_S(p)$ for some set of primes $S$ of a given type.  For instance, take $p=3,g=2,$ and $Z_1 = Z_2 = 1+3\Z_3$.  Then $G_S(p)$ is a group with abelianization $(\Z/3\Z)^2$, and the image of tame inertia at each prime in $S$ lies in a conjugacy class $c$ of $\Gamma$ satisfying $c^4 = c$.  But there exists a group $G$ of order $3^5$ which is not generated by any two elements which are conjugate to their fourth powers.  Thus, $G_S(p)$ cannot be isomorphic to $\Gamma$, and any reasonable heuristic should reflect this by assigning probability $0$ to $\Gamma$.  

In order to avoid problems of this kind, we introduce the notion of a ``pro-$p$-group with local data of type $Z$."  As in Definition~\ref{de:azg}, we denote by $\pi$ the natural projection from a group to its abelianization.

%

\begin{defn} A {\em pro-$p$ group with local data of type $Z$} is a balanced pro-$p$ group $\Gamma$ endowed with a $g$-tuple of conjugacy classes $(c_1, \ldots, c_g)$ and an involution $\iota$ in $\Gamma$ (automatically trivial when $p$ is odd), such that:
\begin{itemize}
\item $c_i^z = c_i$ for all $z \in Z_i$;
\item the projections $\pi(c_1), \ldots, \pi(c_g)$ generate $\Gamma^{ab}$;
\item the map $W(Z) \ra \Gamma^{ab}$ sending $(w_1, \ldots, w_g)$ to $\sum_i w_i \pi(c_i)$ is an isomorphism;
\item if $p = 2$, and $\iota_i$ is the unique nontrivial involution in the cyclic subgroup of $\Gamma^{ab}$ generated by $\pi(c_i)$, we have $\sum_{i=1}^g \iota_i = \pi(\iota)$.
\end{itemize}
\end{defn}

For each $\ell \neq p$, we fix a generator $\tau_\ell$ of the tame inertia group of $G_{\Q_\ell}$.  We also fix a complex conjugation $c$ in $G_\Q$.  When $S = (\ell_1, \ldots, \ell_g)$ is a $g$-tuple of primes with type $Z$, the unramified Galois group $G_S(p)$ acquires the structure of pro-$p$-group with local data of type $Z$ in a natural way; namely, $c_i$ is the conjugacy class of the image of $\tau_{\ell_i}$ in $G_S(p)$ and $\iota$ is the image of $c$.


Finite $p$-groups with local data of type $Z$ constitute the objects of a category $\CC_{p,Z}$ whose morphisms are just homomorphisms of groups preserving $c_1, \ldots, c_g$ and $\iota$.  Just as the category of finite abelian $p$-groups is the ``natural" home of the $p$-part of the class group of a quadratic imaginary field, a finite $p$-group with inertia data is the natural home of $G_S(p)$.  Thus, the Cohen-Lenstra philosophy would suggest that the probability that $G_S(p)$ and $(\Gamma,(c_1, \ldots, c_g), \iota)$ are isomorphic in $\CC_{p,Z}$ should be proportional to $|\Aut_{\CC_{p,i}} (\Gamma,(c_1, \ldots, c_g), \iota)|^{-1}$.  The probability that $G_S(p)$ is isomorphic to $\Gamma$ as a group is then obtained by summing over all $\Aut(\Gamma)$-orbits of inertia data on $\Gamma$.  This yields precisely the prediction that  $P(Z,\Gamma,X)$ is proportional to $A_Z(\Gamma)/|\Aut(\Gamma)|$.

In particular, the fact that the group $G_S(p)$ carries local data of type $Z$ shows that Heuristic~\ref{he:main} passes one easy plausibility check:

\begin{prop}  If $A_Z(\Gamma) = 0$, there is no $g$-tuple $S$ of type $Z$ such that $G_S(p)$ is isomorphic to $\Gamma$.
\end{prop}

\begin{rem} This justification also suggests obvious refinements of Heuristic~\ref{he:main} taking local data into account;  for instance, when $p=2$, one can ask for what proportion of $g$-tuples $S$ there is an isomorphism $G_S(p) \cong \Gamma$ taking complex conjugation to a given conjugacy class of involutions in $\Gamma$.  Furthermore, one could speculate that wildly ramified extensions can be brought into the picture as well by adding to $(c_1, \ldots, c_g),\iota$ the data of a homomorphism $G_{\Q_p} \ra \Gamma$.  This would bring us closer to the heuristics suggested by Bhargava~\cite{bhar:massformulae} and Roberts~\cite{robe:wildpartitions} for the number of extensions of global fields with specified Galois groups and restrictions on local behavior.   On the other hand, it takes us farther away from the analogy with function fields over finite fields that we explore in the next section.
 \end{rem}

The class group of a number field is always finite, so the probability distribution is determined by the fact that the sum of the mass of all finite groups is $1$.  In the present situation, we have no such assurance; for many choices of the type $Z$, $G_S(p)$ can be either finite or infinite, so we have no principled way of choosing one probability distribution on finite $p$-groups among all those proportional to $A_Z(\Gamma)/|\Aut(\Gamma)|$.  We address this problem in the next section.

\section{Justification 2:  Random pro-$p$ braids and the analogy with function fields}

In this section we give an alternative (though certainly related) justification for Heuristic~\ref{he:main}, based on the analogy between $\Z$ and the ring $\F_q[t]$ (or, more generally, the coordinate ring of an affine algebraic curve over a finite field.)

\subsection{The action of Frobenius on the fundamental group of an affine curve}

Let $\F_q$ be a finite field of characteristic other than $p$, let $C/\F_q$ be a smooth projective algebraic curve, let $x_\infty$ be a point of $C(\F_q)$, and let $U/\F_q$ be an open subscheme of $C$ not containing $x_\infty$.  Then the \'{e}tale fundamental group $\pi_1(U)$ fits into an exact sequence
\beq
1 \ra \pi_1(U_{\Fqbar}) \ra \pi_1(U) \ra G_{\F_q} \ra 1
\eeq
in which the second map admits a section $s$ attached to any choice of tangential base point at $x_\infty$.  

Note that the $p$-cyclotomic character $\chi_p: \pi_1(U) \ra \Z_p^*$ factors through the projection to $G_{\F_q}$ and sends Frobenius to $q$.  In particular, the image of $\chi_p$ is procyclic.  The $p$-cyclotomic character of $G_\Q$ is surjective on $\Z_p^*$; in particular, when $p=2$, the image is not procyclic.  In order to make our analogy as precise as possible, we assume from now on that $p$ is odd and that $q$ generates $\Z_p^*$.  We will return in section~\ref{ss:p2} to the case $p=2$.

\begin{rem}  The hypothesis that $q$ generates $\Z_p^*$ is natural when drawing an analogy with $\Q$;   to generate conjectures about maximal pro-$p$ extensions with restricted ramification of larger number fields $K$ would require a different hypothesis on $q$ reflecting the abelian extensions of $\Q$ contained in $K$.  In particular, the presence of roots of unity in $K$ ought to change the statistical distribution of the Galois groups, even in the abelian cases considered by Cohen, Lenstra, and Martinet.  Such a discrepancy from Cohen-Lenstra predictions has in fact been observed numerically by Malle~\cite{mall:cl} in the case of number fields and Rozenhart~\cite{roze:phd} in the case of function fields.  Forthcoming work of Garton~\cite{gart:thesis} will explain how one should modify the Cohen-Lenstra-Martinet heuristics in the presence of roots of unity, based on a function-field argument like the one we present in the present paper.
\end{rem}

Let $G_\infty$ be a decomposition group of $\pi_1(U)$ at $x_\infty$.  Write $I_\infty$ for the inertia subgroup of $G_\infty$.  Note that the section $s: G_{\F_q} \ra \pi_1(U)$ factors through $G_\infty$. 


Write $G_U(p)$ for the quotient of $\pi_1(U)$ by the normal subgroup generated by $G_\infty$.  This is the Galois group of the maximal \'{e}tale cover of $U$ which is totally split at $x_\infty$.  (Note that, in the number field case, the maximal pro-$p$ extension of $\Q$ unramified outside $S$ is indeed totally split at $\infty$, because $p$ is odd here.)



The geometric fundamental group $\pi_1(U_{\Fqbar})$ is isomorphic to a free profinite group on $N$ generators, where $N+1$ is the number of punctures of $U_{\Fqbar}$.  Write $F$ for the quotient of $\pi_1(U_{\Fqbar})$ by the normal subgroup generated by $I_\infty$.  The group $s(G_{\F_q})$ acts on $\pi_1(U_{\Fqbar})$ by conjugation, and this action descends to an action on $F$.  Since $G_{\F_q}$ is procyclic, we can describe this action by specifying the action on $F$ of the Frobenius $\Frob_q$ in $s(G_{\F_q})$, which is an automorphism $\alpha \in \Aut(F)$.

The group $F$ is freely generated by $N$ elements $x_1, \ldots, x_N$, each one a generator of tame inertia at a puncture, subject to the single relation $x_1 \ldots x_N = 1$.  The automorphism $\alpha$ has the special property that it sends each $x_i$ to a conjugate of $x_j^q$ for some $j$.  The automorphisms of $F$ with this property lie in the {\em pro-$p$ braid group} $B_N \subset \Aut(F)$, which we define in the next section.


Now $G_U(p)$ is precisely the quotient of $F$ by relations $\alpha(x_i) = x_i, \forall i$.  In particular, $G_U(p)$ is determined by the automorphism $\alpha$.  The main idea driving the heuristics in this paper is that $\alpha$ should be a {\em random} element drawn from a certain subset of the profinite group $B_N$ under Haar measure.  In the following sections we describe the consequences of this heuristic for the behavior of $G_U(p)$.

\begin{rem}  The Cohen-Lenstra heuristics can also be recovered from this point of view, as was first observed by Friedman and Washington~\cite{frie:friedmanwashington}; the $p$-part of the class group of a random quadratic imaginary field is analogous to the $\F_q$-rational points in $\Jac(X)[p]$, where $X/\F_q$ is a hyperelliptic curve of large genus.  This group, in turn, is the cokernel of $\gamma - 1$, where $\gamma \in \GSp_{2g}(\Z_p)$ is a matrix representing the action of Frobenius on $H_1(X,\Z_p)$.  In fact, Frobenius lies in the coset $\GSp_{2g}^q(\Z_p)$ of $\Sp_{2g}(\Z_p)$ consisting of matrices which scale the symplectic form by $q$.  Taking $\gamma$ to be a random element of $\GSp_{2g}(\Z_p)$ yields precisely the Cohen-Lenstra heuristics.

Friedman and Washington took $\gamma$ to be a random element of $\GL_{2g}(\Z_p)$; the observation that the correct set over which to average is $\GSp_{2g}^q(\Z_p)$ is due to Achter~\cite{acht:cohenlenstra}.  

In fact, as in \cite{acht:cohenlenstra}, these heuristic arguments can often be turned into proofs that heuristics of Cohen-Lenstra type are true over function fields ``in the large $q$ limit."  In this connection, see also the work of Katz and Sarnak on the relation between random $p$-adic matrices and the distribution of zeroes of $L$-functions~\cite{katz:katzsarnak}.
\end{rem}

\subsection{The pro-$p$ braid group and random pro-$p$ groups}

Let $p$ be an odd prime, and let $F$ be the pro-$p$ group generated by elements $x_1, \ldots, x_N$ subject to the single relation $x_1 \ldots x_N = 1$.   We define the {\em pure pro-$p$ braid group} $P_N$ to be the subgroup of $\Aut(F)$ consisting of automorphisms $\alpha$ such that $\alpha(x_i) \sim x_i$ for all $i$.  For each $(q,\sigma) \in  \Zp^* \times S_N$ we denote by $B_N(q,\sigma)$ the set of automorphisms $\alpha$ of $F$ such that
\beq
\alpha(x_i) \sim  x_{\sigma(i)}^q
\eeq
for all $i$ and define the {\em pure pro-$p$ braid group} $P_N$ to be $B_N(1,1)$.  Then the pro-$p$ braid group $B_N$ is defined by
\beq
B_N = \bigcup_{q,\sigma} B_N(q,\sigma)
\eeq

The pro-$p$ braid group fits into an exact sequence
\beq
1 \ra P_N \ra B_N \ra \Z_p^* \times S_N \ra 1
\eeq
where the preimage of $(q,\sigma) \in S_N \times \Z_p^*$ is $B_N(q,\sigma)$.

\begin{rem}  The definition of the pro-$p$ braid group is due to Ihara~\cite{ihar:annals86}.  More precisely, his pro-$p$ braid group on $N$ strands is the image of our $B_N$ in $\Out(F)$.
\end{rem}

\begin{rem} The natural map $B_N \ra \Z_p^*$ is surjective, so $B_N(q,\sigma)$ is nonempty; but we are not aware of any purely group-theoretic proof of this fact.  Rather, one observes (again following Ihara) that when $F$ is identified with the pro-$p$ geometric fundamental group of an $N$-punctured genus $0$ algebraic curve over $\Q$, the resulting action of $G_\Q$ on $F$ induces a map $G_\Q \ra B_N$ whose composition with the projection to $\Z_p^*$ is the cyclotomic character.
\end{rem}

We are now ready to define the main object of study of the present paper:  namely, a ``random pro-$p$ group" $\G(q,\sigma)$.  

\begin{defn}  For $N \geq 1$ and $\sigma \in S_N$, we denote by $\G(q,\sigma)$ the quotient of $F$ by the relations $\alpha(x_i) = x_i, i = 1 \ldots N$, where $\alpha$ is a random element of $B_N(q,\sigma)$ in Haar measure.
\end{defn}

\subsection{Statistical properties of $\G(q,\sigma)$}

Write $k_1, \ldots, k_g$ for the lengths of the cycles of $\sigma$, and choose representatives $i_1, \ldots, i_g$ of the cycles.  Write $g_j$ for the image of $x_{i_j}$ in $\G(q,\sigma)$.  Then one sees that $g_j \sim g_j^{q^{k_j}}$.  We also note that
\beq
\G(q,\sigma)^{ab} = \sum_{j=1}^g \Z_p / (q^{k_g} - 1)\Z_p.
\eeq
We denote this finite abelian group by $W(q,\sigma)$.

\begin{prop} Let $g$ be the number of cycles of $\sigma$.  Then  $\G(q,\sigma)$ is a $g$-generated, $g$-related pro-$p$ group.
\end{prop}

\begin{proof} $\G(q,\sigma)$ is clearly generated by the $g$ elements $g_j$, and is specified by the $g$ relations $\alpha(x_{i_j}) x_{i_j}^{-1}$.  So $h_2(\G(q,\sigma),\F_p) \leq g$; since $\G(q,\sigma)^{ab}$ is finite, we also have $h_2(\G(q,\sigma),\F_p) \geq g$.
\end{proof}

Let $\Gamma$ be some finite $g$-generated $p$-group, and let $c_1, \ldots, c_g$ be conjugacy classes in $\Gamma$ such that
\begin{itemize}
\item $c_1, \ldots, c_g$ generate $\Gamma$;
\item $c_i^{q^{k_i}} = c_i$ for all $i$.
\end{itemize}

We denote by $(\Gamma,\cc)$ the data of a pro-$p$ group $\Gamma$ together with a $g$-tuple of conjugacy classes as above.  Given a permutation $\sigma$ together with a set $i_1, \ldots, i_g$ of cycle representatives for $\sigma$, we write
\beq
\Epi(\G(q,\sigma),(\Gamma,\cc))
\eeq
for the set of surjective homomorphisms $\G(q,\sigma) \ra \Gamma$ which send $g_j$ to the conjugacy class $c_j$ for each $j$.  (The notation is slightly misleading insofar as  $\Epi(\G(q,\sigma),(\Gamma,\cc))$ depends not only on $\sigma$ but on the choice of cycle representatives, but this ambiguity will not concern us.)

Similarly, we write $\Epi(F,(\Gamma,\cc))$ for the set of surjective homomorphisms $F \ra \Gamma$ which send $x_{\sigma^m(i_j)}$ to the conjugacy class $c_j^{q^m}$ for each $j$.  In particular, the pullback of an element of $\Epi(\G(q,\sigma),(\Gamma,\cc))$ to $F$ always lies in $\Epi(F,(\Gamma,\cc))$.  

The set $\Epi(F,(\Gamma,\cc))$ carries an action of $P_N \subset \Aut(F)$ by composition on the left.  Our first aim is to make an educated guess about the transitivity of this action.



\begin{heur}
Suppose $\Gamma$ is balanced.  Then the action of $P_N$ on  $\Epi(F,(\Gamma,\cc))$ is transitive when $N$ is sufficiently large relative to $g$.  \label{he:pure}
\end{heur}

\begin{rem}  A transitivity theorem of this kind for the action of the full braid group has been proven by Fried and V\"{o}lklein~\cite[Appendix]{frie:frvo}.  More precisely, the orbits of the full braid group are identified with a quotient of the Schur multiplier $H_2(\Gamma,\Z)$; in this case, the Schur multiplier is trivial because $\Gamma$ is balanced. See also Lemma 4.10 below.
\end{rem}




We are now ready to present the second justification for Heuristic~\ref{he:main}.

\begin{prop} Suppose $\Gamma$ is a balanced finite $p$-group with abelianization isomorphic to $W(q,\sigma)$, and assume Heuristic~\ref{he:pure}.  Then the probability that $\G(q,\sigma)$ is isomorphic to $\Gamma$ is $A_Z(\Gamma)/|\Aut(\Gamma)|$. 
\label{pr:expqsigma}
\end{prop}

\begin{proof}  The set  $\Epi(\G(q,\sigma),(\Gamma,\cc))$ is the set of $\alpha$-fixed points on $\Epi(F,(\Gamma,\cc))$, where $\alpha$ is a random element of the coset $B_N(q,\sigma)$ of $P_N$ in $B_N$.  By Burnside's lemma, the average number of fixed points of $\gamma$ is the number of fixed points of $(q,\sigma) \in \Z_p^* \times S_N$ in its action on the $P_N$-orbits on $\Epi(F,(\Gamma,\cc))$.  By Heuristic~\ref{he:pure}, there is only one such orbit.  We conclude that
\beq
\E |\Epi(\G(q,\sigma),(\Gamma,\cc))| = 1.
\eeq

\begin{lem}
Every surjection from $\G(q,\sigma)$ to $\Gamma$ is an isomorphism.
\end{lem}

\begin{proof}
The long exact sequence in group cohomology yields an exact sequence
\beq
0 \ra H^1(\Gamma,\Qp/\Zp) / p H^1(\Gamma,\Qp/\Zp) \ra H^2(\Gamma,\Z/p\Z) \ra H^2(\Gamma,\Qp/\Zp)[p]  \ra 0.
\eeq
The first two terms have dimension $g$, so the third term vanishes; hence, so does $H^2(\Gamma,\Qp/\Zp)$.

Write $K$ for the kernel of the map $\G(q,\sigma) \ra \Gamma$.  The inflation-restriction sequence gives
\beq
H^1(\Gamma,\Qp/\Zp) \ra H^1(\G(q,\sigma),\Qp/\Zp) \ra H^1(K,\Qp/\Zp)^\Gamma \ra H^2(\Gamma,\Qp/\Zp) = 0.
\eeq

Our hypothesis on $\Gamma^{ab}$ implies that the first map is an isomorphism; we conclude that $H^1(K,\Qp/\Zp)^\Gamma$ $ = 0$, which implies that $K$ is trivial, since a nontrivial $K$ would have a $\Z/p\Z$-quotient on which $\Gamma$ acts trivially.
\end{proof}


Summing over all $\cc$ of type $Z$, we find that the expected number of isomorphisms from $G(q,\sigma)$ to $\Gamma$ is exactly $A_Z(G)$.  The number of such isomorphisms is either $0$ or $|\Aut(\Gamma)|$; the desired result follows.
\end{proof}

We now explain how Heuristic~\ref{he:main} follows from Proposition \ref{pr:expqsigma}.  The $N$ punctures of the affine curve $U/\F_q$ (excepting $x_\infty$) can be thought of as $g$ closed points; we refer to this set of points as $S$ and define the type $S$ as in the first section.  The hypothesis that $S$ is of some specified type $Z$ can be thought of as a condition on $\sigma$; namely, that $W(q,\sigma) \cong W(Z)$.  So the probability that $G_U(p) \cong G$ can be heuristically estimated as the probability that $\G(q,\sigma) \cong G$, conditional on the fact that $W(q,\sigma) \cong W(Z)$.  By Proposition~\ref{pr:expqsigma}, This probability is $A_Z(\Gamma)/|\Aut(\Gamma)|$ for each such $\sigma$.

More generally, the argument of this section suggests the following:   Let $\Gamma$ be a $p$-group such that the action of the pure braid group $P_N$ on $\Epi(F,(\Gamma,\cc))$ is transitive; then the expected number of epimorphisms from $G_S(p)$ to $\Gamma$ is $A_Z(\Gamma)$.  (We emphasize that we do not presently know any examples where the action of $P_N$ is intransitive.)

In particular, suppose $\Gamma$ is a group of $p$-class $k$, with the property that the existence of a surjection from a balanced group $G$ to $\Gamma$ implies that $\Gamma$ is the quotient of $G$ by the $(k+1)$-th term in its lower $p$-central series.  (We will encounter such examples in the data presented in the following sections.)  Then the probability that the $p$-class $k$ quotient of $G(q,\sigma)$ is isomorphic to $\Gamma$ should be $A_Z(\Gamma)/|\Aut(G)|$.

\subsection{The case $p=2$}

\label{ss:p2}

The geometric analogy in the case $p=2$ is somewhat more difficult to justify, since there is no choice of $q$ making the cyclotomic character $\chi: G_{\F_q} \ra \Z_2^*$ surjective.  Our belief in Heuristic~\ref{he:main} in this case stands on three legs:
\begin{itemize}
\item The argument via Justification 1 that $\lim P(Z,\Gamma,X)$ should be proportional to $A_Z(\Gamma)/|\Aut(\Gamma)|$ for all $p$;
\item The argument via Justification 2 that, in case $p$ is odd, the constant of proportionality should be $1$;
\item The case $p=2$ is the easiest to test experimentally; as we shall see in the following section, Heuristic~\ref{he:main} appears to agree very well with the experimental data and with provable asymptotics when such are available.
\end{itemize}

\section{Evidence}

\label{s:evidence}

Let $p$ be an odd prime and $S$ a set of $g$ primes that are all $1 \pmod p$. Then the Galois group of the maximal $p$-extension of $\Q$ unramified outside $S$ is a pro-$p$ group with generator rank $g$ and relator rank $g$. The same result holds for $p=2$ if we allow ramification at $\infty$. 

Given a pro-$p$ group $\Gamma$ with generator rank and relator rank both equal to $g$, we consider the set of $g$-tuples of primes, all $1 \pmod p$, such that the corresponding Galois group is isomorphic to $\Gamma$. In all cases considered so far, this set appears to have density predicted by Heuristic~\ref{he:main}.  In some cases this result on the density can be proven, whereas in others we give experimental evidence.

\subsection{Groups with abelianization $\Z/p\Z \times \Z/p\Z$ ($p$ odd)}

In this case, $S$ consists of two primes, say $q,r$, that are $1 \pmod p$ but not $1 \pmod {p^2}\ (\ast)$. 

(i) The smallest $2$-generator $2$-relator $p$-group $\Gamma$ is  unique of order $p^3$ (denoted SmallGroup($p^3,4$) in the MAGMA database). As shown on p.129 of \cite{koch}, the Galois group of the maximal $p$-extension of $\Q$ unramified outside $q,r$ is isomorphic to $\Gamma$ if and only if $q$ is not a $p$th power $\pmod r$ or $r$ is not a $p$th power $\pmod q$. Among ordered pairs of primes satisfying $(\ast)$, the natural density of such pairs is then, by Chebotarev, $1-1/p^2$. 

As for the heuristic, there are $p^2-1$ conjugacy classes that are outside the Frattini subgroup of $\Gamma$ and that are closed under taking $(p+1)$th powers. Working in $\Z/p\Z \times \Z/p\Z$, we see that $p(p+1)(p-1)^2$ ordered pairs of such conjugacy classes generate $\Gamma^{ab}$ and hence $\Gamma$. On the other hand, the automorphism group of $\Gamma$ has order $(p-1)p^3$ and so $A(\Gamma)/|\text{Aut}\  \Gamma| = 1-1/p^2$.

(ii) We focus on $p=3$. The next smallest $2$-generator $2$-relator $3$-group with abelianization $\Z/3\Z \times \Z/3\Z$ and nonzero probability of arising is a unique group $\Gamma$ of order $3^5$. There is another $2$-generator $2$-relator group of order $3^5$, but it is not generated by elements conjugate to their $4$th power. For such groups our conjecture is trivially true, since it follows that $A(\Gamma) = 0$ and also that $\Gamma$ lacks elements that would play the role of tame inertia generators.

Suppose that $q,r$ are primes that are $1 \pmod 3$ but not $1 \pmod 9$. By the method of Boston and Leedham-Green~\cite{BLG}, the Galois group of the maximal $p$-extension of $\Q$ unramified outside $q,r$ is isomorphic to $\Gamma$ if and only if two of the Sylow $3$-subgroups of the ray class groups of modulus $qr$ of its $4$ cubic subextensions are isomorphic to $\Z/3\Z \times \Z/3\Z \times \Z/3\Z$ and two isomorphic to $\Z/3\Z \times \Z/9\Z$.  This implies in particular that $q$ is a cube mod $r$ and $r$ is a cube mod $q$, a condition which is satisfied for $1/9$ of all pairs $(q,r)$.  Heuristic~\ref{he:main} asserts that $2/9$ of these will have $G_{\{q,r\}}(3)$ isomorphic to $\Gamma$, since $A(\Gamma)/|\text{Aut}\  \Gamma| = 2/81.$

We tested $58$ pairs with $q$ a cube mod $r$ and $r$ a cube mod $q$; a MAGMA computation shows that $13$ of these pairs satisfy the condition on ray class groups of subextensions equivalent to $G_{\{q,r\}}(3) \cong \Gamma$.  The empirical density $13/58 \times 1/9 = 0.0249$ compares well with the predicted $2/81 = 0.0247$.

\subsection{Groups with abelianization $\Z/2\Z \times \Z/2\Z$}

In this case, $S$ consists of two primes $q$ and $r$ both congruent to $3 \pmod 4$.  Order $q$ and $r$ such that $q$ is a square mod $r$.  Then, as shown in Theorem 2.1 of Boston-Perry~\cite{BP}, the Galois group $G_{\{q,r\}}(2)$ is semidihedral of order $2^k$, where $2^k$ is the largest power of $2$ dividing $q^2-1$.  In particular, $G_{\{q,r\}}(2)$ is determined by the type of $(q,r)$.   This agrees with the heuristic, which gives  $A_Z(\Gamma)/|\text{Aut}\  \Gamma| = 1$ for each semidihedral group $\Gamma$. 

\subsection{Groups with abelianization $\Z/2\Z \times \Z/4\Z$}

In this case the two primes $q,r$ are $3 \pmod 4$ and $5 \pmod 8$. Let us suppose further that $q$ is $3 \pmod 8$, which fixes the type. 

(i) The smallest such $2$-relator group is the modular group of order $16$. The Galois group of the maximal $2$-extension unramified outside $\{q,r\}$ is isomorphic to this group if and only if $q$ is not a square $\pmod r$. This occurs with probability $1/2$, as Heuristic~\ref{he:main} predicts.

(ii) The next smallest such $2$-relator group with nonzero probability of arising has order $128$ (SmallGroup($128,87$) in the MAGMA database). As shown in Boston-Perry~\cite{BP}, this arises if and only if $q$ is a $4$th power $\pmod r$. This occurs with probability $1/4$ and agrees with Heuristic~\ref{he:main}.

(iii) The next smallest such $2$-relator groups with nonzero probability are the two groups $\Gamma_1, \Gamma_2$ of order $2^{19}$ found by Boston and Leedham-Green~\cite{BLG}. They showed there that the Galois group of the maximal $2$-extension unramified outside $\{q,r\}$ is isomorphic to $\Gamma_1$ or $\Gamma_2$ if and only if $q$ is a square but not a $4$th power mod $r$ and the Sylow $2$-subgroup of the ray class group of modulus $qr$ of $\Q(\sqrt{-qr})$ is isomorphic to $\Z/2 \times \Z/16$. The heuristic suggests that each group should arise $1/16$ of the time, so that this case should arise $1/8$ of the time. 

The experimental evidence agrees - we test it by letting $q=5$ and $p$ run through all primes $< 100000$ that are $19 \pmod {40}$ (so that, as assumed, $p$ is $3 \pmod 8$ and a square but not a $4$th power $\pmod q$). In general, the Sylow $2$-subgroup of the ray class group is $\Z/2\Z \times \Z/2^n\Z$ with $n \geq 4$. The cases $n=4, 5, 6, 7$ appear respectively $301, 151, 74, 42$ times, which clearly suggests that half of the $1/4$ of cases remaining from (i) and (ii) above have $n=4$.

\subsection{Groups with abelianization $\Z/4\Z \times \Z/4\Z$}

\includegraphics[height=240mm]{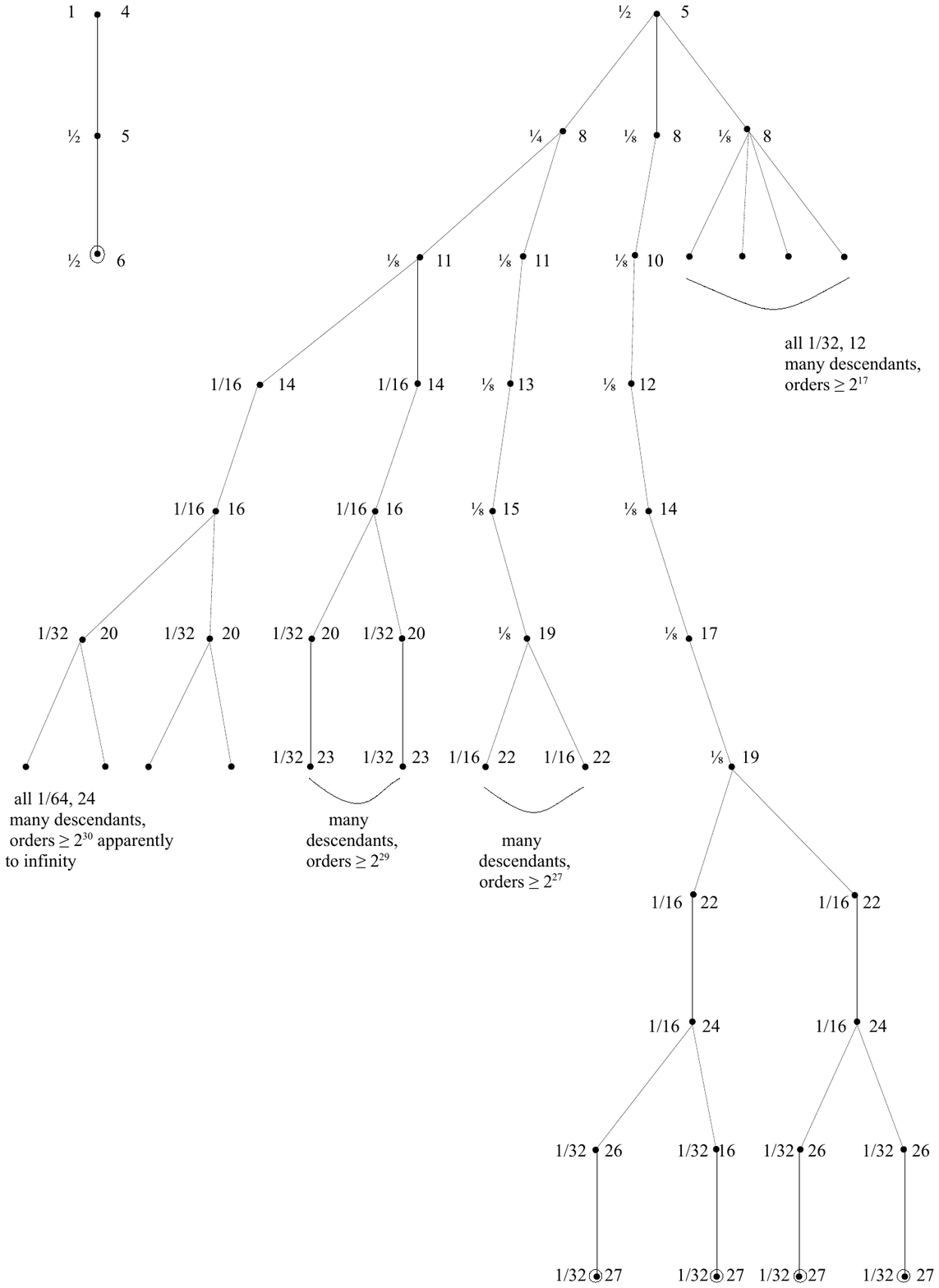}

Here the two primes $q,r$ are $5 \pmod 8$, which fixes the type. 

The figure gives a subtree of the O'Brien tree \cite{OBp} of $2$-generator $2$-groups. The vertices of his tree are isomorphism classes of $2$-generator $2$-groups. It has a root, namely $\Z/2\Z \times \Z/2\Z$, and the groups at distance $n-1$ from the root are those of $2$-class $n$. A group of $2$-class $n$ is connected to a group $H$ of $2$-class $n+1$ if and only if it is isomorphic to $H/P_n(H)$, where $P_n(H)$ is the last but one term of its lower $2$-central series. 

Our subtree consists of those groups $G$ that have abelianization $\Z/4\Z \times \Z/4\Z$, have nonzero ``mass" $A(G)$ (its value is given to the left of each vertex), and are $2$-class quotients of some (possibly infinite) $2$-relator pro-$2$ group. We call such groups {\it viable}. This last matter can be detected by checking that the $2$-multiplicator rank of $G$ minus its nuclear rank is at most $2$. If the nuclear rank of $G$ is zero, then it is $2$-relator. In that case it has no descendants and is a terminal vertex. Otherwise, if the nuclear rank of $G$ is $r$, then (except for the root) its immediate descendants all have order $2^r|G|$. This is important for us since, except for the groups of $2$-class $2$, one descendant of a group cannot be a quotient of another and we do indeed (assuming transitivity of the pure braid group action) have a distribution on a rooted tree. Note that, at least under this assumption, one might expect $A(\Gamma)$ to be the number of surjections from $G_{\{q,r\}}(2)$ to $\Gamma$. The exponent of the order of a group is placed to the right of the vertex.

The subtree falls into $4$ parts. By Part I, we mean the isolated twig to the left. Let $\Gamma_i (i = 1,2,3)$ denote the three groups of $2$-class $3$ and order $256$. Parts II, III, and IV will refer to the subtrees with roots $\Gamma_2, \Gamma_3$, and $\Gamma_1$ respectively. Two of the abelianizations of the index $2$ subgroups of each group are isomorphic to $\Z/2\Z \times \Z/4\Z \times \Z/4\Z$, whereas the third is $\Z/2\Z \times \Z/4\Z \times \Z/4\Z$ for $\Gamma_1$ and $\Z/2\Z \times \Z/2\Z \times \Z/8\Z$ for the other two groups. Ray class groups of quadratic fields tell us that $\Gamma_1$ arises if one of $q,r$ is a $4$th power mod the other but not vice versa, which does occur with density $1/4$, and that $\Gamma_2$ or $\Gamma_3$ arises otherwise, occurring with density $1/4$. These match the predicted values.

\subsubsection{Part I}

The smallest $2$-relator group in the subtree has order $64$ (SmallGroup($64,28$) in the MAGMA database) and arises if and only if $q$ is not a square mod $r$. This occurs with probability $1/2$, as predicted by Heuristic~\ref{he:main}.

\subsubsection{Part II}

The nuclear rank of $\Gamma_2$ is $2$, indicating that its descendant tree should not be too complicated. Using MAGMA to produce its viable descendants we find that its subtree is finite, terminating in four groups of order $2^{27}$ and $2$-class $12$. These have not appeared in the literature before. They each have mass $1/32$ according to the heuristic. Among the abelianizations of their index $2$ subgroups is still $\Z/2\Z \times \Z/2\Z \times \Z/8\Z$. On the other hand, the viable descendants of $\Gamma_3$ all have larger abelianizations of index $2$ subgroups. It follows that if neither $q$ nor $r$ is a $4$th power mod the other, then the Galois group of the maximal $2$-extension of $\Q$ unramified outside $\{q,r\}$ is one of these $4$ groups of order $2^{27}$. Moreover, the total mass they carry, namely $1/8$, agrees with our conjecture, which furthermore suggests that each of the $4$ groups should occur equally often as $q,r$ vary.

\subsubsection{Part III}

The descendant tree of $\Gamma_3$ contains groups with nuclear rank $5$ and higher, which makes it prohibitive to calculate to any depth. By the process of elimination above, it and its descendants correspond to $q,r$ that are both $4$th powers mod the other. Since in this case the corresponding ray class groups include one isomorphic to $\Z/2\Z \times \Z/2\Z \times \Z/2^n\Z$ for any $n \geq 4$, there will be infinitely many possible Galois groups, but as conjectured in \cite{Bos1}, the evidence suggests that these groups are all finite.

That evidence came from performing the experiment of looking at millions of pro-$2$ groups with presentation of the form $< x,y | x^a=x^5, y^b=y^5 >$ (as $a$ and $b$ vary through words in $x$ and $y$) and by saving those whose $2$-class quotients appeared to grow unboundedly and whose finite index subgroups had finite abelianization (within reasonable computational limits). All the groups that passed these filters had abelianizations of all index $2$ subgroups isomorphic to $\Z/2\Z \times \Z/4\Z \times \Z/4\Z$; in other words, it appears that whenever $G_{\{q,r\}}(2)$ is infinite, it is a descendant of $\Gamma_1$.  We have established above that all descendants of $\Gamma_2$ are finite; we conjecture that, likewise, all descendants of $\Gamma_3$ are finite.


\subsubsection{Part IV}

The first author also suggested in \cite{Bos1} that all viable $2$-relator descendants of $\Gamma_1$ should be infinite, and that the sequence of orders of their $2$-class quotients should be a particular one (A001461 in Sloane's database of sequences). Our new investigations indicate that this must be modified to say that this is true for exactly one quarter of these cases. 

What happens is that $\Gamma_1$ has two viable descendants, each order $2^{11}$ and carrying mass $1/8$. One of these has two viable descendants, each of order $2^{14}$ and carrying mass $1/16$. Denote by $\Gamma$ the first of these groups of order $2^{14}$. Extensive experiments indicate that every pro-$2$ group with presentation of the form $< x,y | x^a=x^5, y^b=y^5 >$ and an element of order two outside its commutator subgroup (playing the role of complex conjugation) and whose $2$-class $5$ quotient is $\Gamma$
is infinite and has $2$-class tower whose orders match the sequence A001461.  Moreover, it
appears that every descendant of $\Gamma_1$ which is not a descendant of $\Gamma$ is finite.

The descendant tree of the second groups of order $2^{11}$ and $2^{14}$ can be pursued for quite a distance. The experiment above indicates that these subtrees will eventually terminate.

\section{Conservation of mass}

For a fixed prime $p$ and positive integer $g$, consider O'Brien's rooted tree whose nodes are isomorphism classes of finite $g$-generator $p$-groups. We prune this by saving only those groups of a particular type (which fixes their abelianization), that could arise as a $p$-class quotient of a $g$-relator group (their $p$-multiplicator and nuclear ranks differ by at most $g$), and that have nonzero mass. 

It can happen that one such group of a given $p$-class is a quotient of another - for instance, if $F$ is the free pro-$2$ group on $2$ generators, then $F/P_2(F)$ of order $32$ has $\Z/4\Z \times \Z/4\Z$ as a quotient but both are of $p$-class $2$. This tends however to be rare, since if the difference between the $p$-multiplicator rank and nuclear rank of $\Gamma$ equals $g$, the same will be true of its immediate descendants and they will all have the same order, namely $|\Gamma|p^r$, where $r$ is the nuclear rank of $G$ \cite{nover}. This means that (again, given some hypothesis on transitivity of pure braid group actions) we have produced a probability distribution on the rooted tree with root $\Gamma$, by assigning the mass $A_Z(\Gamma)/|\text{Aut}\  \Gamma|$ to each vertex $\Gamma$.

Two interesting questions arise. Is it possible to have a point mass, i.e. an infinite end along which the mass is bounded away from $0$?  Second, does the accumulated mass of FAb groups (meaning that every open subgroup has finite abelianization) or ones without infinite $p$-adic analytic quotients amount to $100\%$? These both have positive answers with respect to the measure given in \cite{Bos2}.

\end{document}